\documentclass[11pt]{article}

\usepackage{amsmath,amssymb,amsthm,amscd}
\usepackage{graphicx}

\def \N {\mathbb{N}}

\def \A {\mathcal{A}}
\def \B {\mathcal{B}}
\def \ol {\overline}
\def \sgp {\operatorname{Sgp}}
\def \mon {\operatorname{Mon}}

\newcommand{\fsa}[1]{\mathcal{#1}}
\newcommand{\defterm}[1]{\textit{#1}}

\DeclareMathOperator{\End}{End}

\newcommand{\displayaut}[1]{\par\medskip\centerline{\includegraphics{#1}}\medskip\noindent\ignorespaces}
\let\emptyword=\varepsilon

\newtheorem{theorem}{Theorem}

\newtheorem{lemma}[theorem]{Lemma}
\newtheorem{corollary}[theorem]{Corollary}

\newtheorem{conjecture}[theorem]{Conjecture}
\newtheorem{question}[theorem]{Open Problem}

\title{Automaton semigroup constructions}
\author{Tara Brough\footnote{School of Mathematics 
and Statistics, Mathematical Institute, North Haugh, St Andrews, Fife KY16 9SS, Scotland.  Email: tara@mcs.st-andrews.ac.uk.} and 
Alan J. Cain\footnote{Centro de Matem\'{a}tica, Faculdade de Ci\^{e}ncias, Universidade do Porto,
Rua do Campo Alegre 687, 4169--007 Porto, Portugal.  Email: ajcain@fc.up.pt}}
\date{}

\begin{document}

\maketitle

\begin{abstract}
The aim of this paper is to investigate whether the class of automaton
semigroups is closed under certain semigroup constructions. We prove
that the free product of two automaton semigroups that contain left
identities is again an automaton semigroup. We also show that the
class of automaton semigroups is closed under the combined operation
of `free product followed by adjoining an identity'. We present an
example of a free product of finite semigroups that we conjecture is
not an automaton semigroup. Turning to wreath products, we consider
two slight generalizations of the concept of an automaton semigroup,
and show that a wreath product of an automaton monoid and a finite
monoid arises as a generalized automaton semigroup in both
senses. We also suggest a potential counterexample that would show
that a wreath product of an automaton monoid and a finite monoid is
not a necessarily an automaton monoid in the usual sense.
\end{abstract}

\section{Introduction}

Automaton groups (that is, groups of automorphisms of labelled rooted trees
generated by actions of automata) arose in the 1980s as interesting
examples having `exotic' properties. Grigor\v{c}uk's infinite periodic
group \cite{grigorchuk_burnside} was the first such example, and it
inspired later ones such as the Gupta--Sidki group
\cite{gupta_burnside}. Since then, a substantial theory has developed;
see, for example, Nekrashevych's monograph \cite{nekrashevych_self} or
one of the surveys by the school led by Bartholdi, Grigorchuk,
Nekrashevych, and \v{S}uni\'{c}
\cite{bartholdi_fractal,bartholdi_branch,grigorchuk_self}.

In recent years, the notion of an automaton semigroup has emerged as a
natural generalization of automaton groups. The basic theory was
outlined by Grigorchuk, Nekrashevych \& Sushchanskii~\cite[esp.~\S 4
  \&~\S7.2]{grigorchuk_automata}. Silva \& Steinberg studied a class
of semigroup generalizing the lamplighter
group~\cite{silva_lamplighter}. Maltcev studied automaton semigroups
arising from Cayley automata (which arise from the Cayley graphs of
finite semigroups) \cite{maltcev_cayley}. These semigroups also formed
part of studies by the second author \cite{c_1auto} and by Mintz
\cite{mintz_cayley}. There have also been studies of algorithmic
problems (see, for instance, \cite{klimann_implementing}), leading to
the recent proof that the finiteness problem is undecidable for
automaton semigroups \cite{gillibert_finiteness}.

A fundamental question has been whether the class of automaton
semigroups is closed under various semigroup constructions. For some
constructions, such as direct products and adjoining a zero or
identity, it is straightforward to prove that the class is closed; see
\cite[\S~5]{c_1auto}. This paper deals with the problems of whether the
class of automaton semigroups is closed under forming semigroup or
monoid free products; and whether the class automaton monoids is closed 
under forming wreath products with finite top semigroup. 
The first of these problems is connected with the even more fundamental
question of whether all finite-rank free groups (which are, after all,
free products of copies of the infinite cyclic group) arise as
automaton groups; only recently has this question been answered
positively \cite{steinberg_automata}.  In contrast, it is relatively
easy to construct all free semigroups (of rank at least $2$) as
automaton semigroups \cite[Proposition~4.1]{c_1auto}.

One problem with constructing counterexamples is that if a semigroup
has the properties that automaton semigroups have generally, such as
residual finiteness \cite[Proposition~3.2]{c_1auto}, then there are no
general techniques for proving it is not an automaton semigroup. For
instance, the proof that the free semigroup of rank $1$ is not an
automaton semigroup is highly specialized and does not seem to
generalize to give any useful strategy
\cite[Proposition~4.3]{c_1auto}.
For research on this topic to develop much further, it will be essential 
to develop techniques for proving that a given semigroup is not an automaton semigroup.

\section{Automaton semigroups}

An \defterm{automaton} $\fsa{A}$ is formally a triple $(Q,B,\delta)$,
where $Q$ is a finite set of \defterm{states}, $B$ is a finite
alphabet of \defterm{symbols}, and $\delta$ is a transformation of the
set $Q \times B$. The automaton $\fsa{A}$ is normally viewed as a
directed labelled graph with vertex set $Q$ and an edge from $q$ to
$r$ labelled by $x \mid y$ when $(q,x)\delta = (r,y)$:
\displayaut{\jobname-1.eps}
The interpretation of this is that if the automaton $\fsa{A}$ is in
state $q$ and reads symbol $x$, then it changes to the state $r$ and
outputs the symbol $y$. Thus, starting in some state $q_0$, the
automaton can read a sequence of symbols
$\alpha_1\alpha_2\ldots\alpha_n$ and output a sequence
$\beta_1\beta_2\ldots\beta_n$, where $(q_{i-1},\alpha_i)\delta =
(q_i,\beta_i)$ for all $i = 1,\ldots,n$.

Such automata are more usually known in computer science as
deterministic real-time (synchronous) transducers. In the field of
automaton semigroups and groups, they are simply called `automata' and
this paper retains this terminology.

Each state $q \in Q$ acts on $B^*$, the set of finite sequences of
elements of $B$. The action of $q \in Q$ on $B^*$ is defined as
follows: $\alpha \cdot q$ (the result of $q$ acting on $\alpha$) is
defined to be the sequence the automaton outputs when it starts in the
state $q$ and reads the sequence $\alpha$. That is, if $\alpha =
\alpha_1\alpha_2\ldots\alpha_n$ (where $\alpha_i \in B$), then $\alpha
\cdot q$ is the sequence $\beta_1\beta_2\ldots\beta_n$ (where $\beta_i
\in B$), where $(q_{i-1},\alpha_i)\delta = (q_i,\beta_i)$ for all $i =
1,\ldots,n$, with $q_0 = q$.

The set $B^*$ can be identified with an ordered regular tree of degree
$|B|$. The vertices of this tree are labelled by the elements of
$B^*$. The root vertex is labelled with the empty word $\emptyword$,
and a vertex labelled $\alpha$ (where $\alpha \in B^*$) has $|B|$
children whose labels are $\alpha\beta$ for each $\beta \in B$. It is
convenient not to distinguish between a vertex and its label, and thus
one normally refers to `the vertex $\alpha$' rather than `the vertex
labelled by $\alpha$'. (Figure~\ref{fig:rootedtree} illustrates the
tree corresponding to $\{0,1\}^*$.)

\begin{figure}[tb]
\centerline{\includegraphics{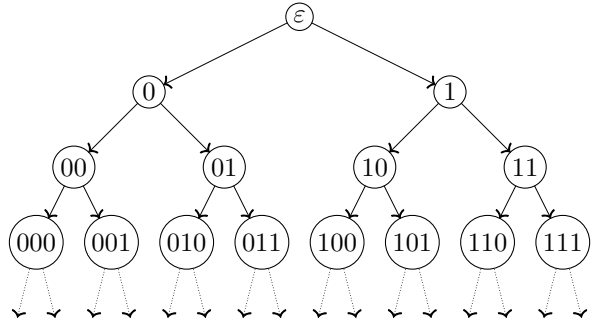}}
\caption{The set $\{0,1\}^*$ viewed as a rooted binary tree.}
\label{fig:rootedtree}
\end{figure}

The action of a state $q$ on $B^*$ can thus be viewed as a
transformation of the corresponding tree, sending the vertex $w$ to
the vertex $w \cdot q$. Notice that, by the definition of the action
of $q$, if $\alpha\alpha' \cdot q = \beta\beta'$ (where $\alpha,\beta
\in B^*$ and $\alpha',\beta' \in B$), then $\alpha \cdot q = \beta
\cdot q$. In terms of the transformation on the tree, this says that
if one vertex ($\alpha$) is the parent of another ($\alpha\alpha'$),
then their images under the action by $q$ are also parent ($\beta$)
and child ($\beta\beta'$) vertices. More concisely, the action of $q$
on the tree preserves adjacency and is thus an endomorphism of the
tree. Furthermore, the action's preservation of lengths of sequences
becomes a preservation of levels in the tree.

The actions of states extends naturally to actions of words: $w =
w_1\cdots w_n$ (where $w_i \in Q$) acts on $\alpha \in B^*$ by
\[
(\cdots((\alpha \cdot w_1)\cdot w_2) \cdots w_{n-1})\cdot w_n.
\]

So there is a natural homomorphism $\phi : Q^+ \to \End B^*$, where
$\End B^*$ denotes the endomorphism semigroup of the tree $B^*$. The
image of $\phi$ in $\End B^*$, which is necessarily a semigroup, is
denoted $\Sigma(\fsa{A})$.

A semigroup $S$ is called an \defterm{automaton semigroup} if there
exists an automaton $\fsa{A}$ such that $S \simeq \Sigma(\fsa{A})$.

It is often more convenient to reason about the action of a state or
word on a single sequence of infinite length than on sequences of some
arbitrary fixed length. The set of infinite sequences over $B$ is
denoted $B^\omega$. The infinite sequence consisting of countably many
repetitions of the finite word $\alpha \in B^*$ is denoted
$\alpha^\omega$. For synchronous automata, the action on infinite
sequences determines the action on finite sequences and \emph{vice
  versa}.

The following lemma summarizes the conditions under which two words
$w$ and $w'$ in $Q^+$ represent the same element of the automaton
semigroup. The results follow immediately from the definitions, but are
so fundamental that they deserve explicit statement:

\begin{lemma}
\label{lem:eq}
Let $w,w' \in Q^+$. Then the following are equivalent:
\begin{enumerate}
\item $w$ and $w'$ represent the same element of $\Sigma(\fsa{A})$;
\item $w\phi = w'\phi$;
\item $\alpha\cdot w = \alpha\cdot w'$ for each $\alpha \in B^*$;
\item $w$ and $w'$ have the same actions on $B^n$ for every $n \in
  \N^0$;
\item $w$ and $w'$ have the same actions on $B^\omega$.
\end{enumerate}
\end{lemma}

Generally, there is no need to make a notational distinction between
$w$ and $w\phi$. Thus $w$ denotes both an element of $Q^+$ and the
image of this word in $\Sigma(\fsa{A})$. In particular, one writes `$w
= w'$ in $\Sigma(\fsa{A})$' instead of the strictly correct `$w\phi =
w'\phi$'. With this convention, notice that $Q$ generates
$\Sigma(\fsa{A})$.

\section{Free products}

The \emph{semigroup free product} (henceforth \emph{free product}) of two semigroups $S = \sgp \langle X_1 \mid R_1\rangle$ and 
$T = \sgp \langle X_2 \mid R_2\rangle$, denoted $S\star T$, is the semigroup with presentation $\sgp \langle X_1\cup X_2 \mid R_1\cup R_2\rangle$.
If $S = \mon \langle X_1 \mid R_1\rangle$ and $T = \mon \langle X_2 \mid R_2\rangle$ are monoids, then the \emph{monoid free product}
$S *_{\mon} T$ of 
$S$ and $T$ is the monoid with presentation $\mon \langle X_1\cup X_2 \mid R_1\cup R_2 \rangle$.  The difference is that the identities of $S$ and $T$
are amalgamated in the monoid free product, but not in the semigroup free product (even if present).

We will show that if two automaton semigroups both contain left
identities, then their free product is an automaton semigroup.  Also,
if we take the free product of any two automaton semigroups \emph{and
  adjoin an identity}, then the result is an automaton semigroup.  We
conjecture that the class of automaton semigroups is \emph{not} closed
under taking free products, and we suggest a possible counterexample.

\begin{theorem}
Let $S$ and $T$ be automaton semigroups, each containing a left identity.  Then $S\star T$ is an automaton semigroup.
\end{theorem}

[There are two dual notions of automaton semigroups, depending on whether the states act on the right (as in this paper) or on the left.
Using the dual notion would give a version of this proposition for semigroups with right identities.]

\begin{proof}
Let $l_S$ and $l_T$ be distinguised left identities in $S$ and $T$ respectively.
Let $\A_1 = (Q_1,A,\delta_1)$ and $\A_2 = (Q_2,B,\delta_2)$ be automata with $\Sigma(\A_1) = S$ and $\Sigma(\A_2) = T$.
We may assume that $Q_1\subset S$ and $Q_2\subset T$ and that $Q_1$ and $Q_2$ contain $l_S$ and $l_T$ respectively.

We construct an automaton $\A$ with $\Sigma(\A) = S\star T$ using $\A_1$ and $\A_2$.
First let $A^\circ = \{a^\circ\mid a\in A\}$ and $B^\circ = \{b^\circ \mid b\in B\}$.
Let $\A = (Q, C, \delta)$ with $Q = Q_1\cup Q_2$, $C = A\cup B\cup A^\circ \cup B^\circ \cup \{\$,\#,\$^\circ,\#^\circ\}$
and $\delta: Q\times C\rightarrow Q\times C$ given by $\delta(q,x^\circ) = (q,x^\circ)$ for all $q\in Q$ and
$x\in A\cup B\cup \{\$,\#\}$ and
\begin{align*}
\delta(s,a)       &= \delta_1(s,a)    &  \delta(t,b)       &= \delta_2(t,b)\\
\delta(s,b)       &= (s,b^\circ)      &  \delta(t,a)       &= (t,a^\circ)\\
\delta(s,\#)      &= (s,\#^\circ)     &  \delta(t,\$)      &= (t,\$^\circ)\\
\delta(s,\$)      &= (l_T,\$)         &  \delta(t,\#)      &= (l_S,\#)
\end{align*}
for $a\in A$, $b\in B$, $s\in Q_1$ and $t\in Q_2$.

We will henceforth refer to the subautomaton consisting of states from $Q_i$ as $\A'_i$ ($i = 1,2$).
We will refer to symbols in $A^\circ \cup B^\circ \cup \{\$^\circ,\#^\circ\}$ as `marked symbols'.
All states act trivially on marked symbols.  Symbols are marked in order to show that they are finished with
and should never be altered further.

The subautomaton $\A'_1$ acts the same as $\A_1$ on symbols from $A$, marks symbols from $B\cup \{\#\}$ without changing state,
and transitions to the state $l_T$ in $\A_2$ on input $\$$, outputting $\$$.
The construction is symmetric in the pairs $(\A_1,\A_2)$, $(\A'_1,\A'_2)$, $(A,B)$ and $(\$,\#)$.

Any $\alpha \in C^*$ is a prefix of some $\beta = u_1 \$ v_1 \# \ldots u_i \$ v_i \# \ldots \in C^\omega$ with
${u_j\in (C\setminus \{\$\})^*}$ and $v_j\in (C\setminus \{\#\})^*$.
We have  \[\beta\cdot l_S = (u_1\cdot l_S) \$ (v_1\cdot l_T) \# \ldots (u_i\cdot l_S) \$ (v_i\cdot l_T) \# \ldots.\]

Since $l_S$ and $l_T$ act as idempotents on $C\setminus \{\$\}$ and $C\setminus \{\#\}$ respectively,
this shows that $l_S$ (and, by symmetry, $l_T$) is an idempotent in $\Sigma(\A)$.
Hence, for $\beta$ as above and $s_1\ldots s_k\in Q_1$, we have
\begin{align*}
\beta\cdot s_1\ldots s_k &= (u_1\cdot s_1) \$ (v_1\cdot l_T) \# \ldots (u_i\cdot l_S) \$ (v_i\cdot l_T) \# \ldots \cdot s_2\ldots s_k\\
&= (u_1\cdot s_1\ldots s_k) \$ (v_1\cdot l_T) \# \ldots (u_i\cdot l_S) \$ (v_i\cdot l_T) \# \ldots,
\end{align*}
so the action of $Q_1^+$ on $C^*$ depends only on its action on $A^*$, and hence the subsemigroup of $\Sigma(\A)$ generated by $Q_1^+$
is isomorphic to $S$.  By symmetry, the subsemigroup generated by $Q_2^+$ is isomorphic to $T$.
Hence, by induction, any two words in $Q^+$ representing the same element of $S\star T$ will act the same on $C^*$,
so $\A$ defines an action of $S\star T$ on $C^*$.

Now in order to conclude that $S(\A) = S\star T$, we need to show that the action of $S\star T$ on $C^\omega$ defined by $\A$ is faithful.
We can distinguish most pairs of strings by their actions on $(\$\#)^\omega$ and $(\#\$)^\omega$.
For $w\in Q^+$, define $x_w = (\$\#)^\omega\cdot w$ and $y_w = (\#\$)^\omega\cdot w$.
Let $\ol{w}$ be the reduced word in $(S\cup T)^*$ corresponding to $w$.
For $w\in Q_1 Q^+$, we have
\[x_w = \begin{cases}(\$^\circ\#^\circ)^{k-1} \$^\circ (\#\$)^\omega & \text{if } l(\ol{w}) = 2k,\\
(\$^\circ\#^\circ)^k (\$\#)^\omega & \text{if } l(\ol{w}) = 2k+1,
\end{cases}
\]
while
\[y_w = \begin{cases} (\#^\circ\$^\circ)^k (\#\$)^\omega & \text{if } l(\ol{w}) = 2k,\\
(\#^\circ\$^\circ)^k \#^\circ (\$\#)^\omega & \text{if } l(\ol{w}) = 2k+1.
\end{cases}
\]
Together with the corresponding statements for $w\in Q_2 Q^+$, this tells us that the only pairs elements of
$S\star T$ which cannot be distinguished by $x_w$ and $y_w$ are those of the same reduced length, beginning with a letter from the same $Q_i$.

Finally, let $w$ and $w'$ be words in $Q^+$ of the same reduced length, representing different elements of $S\star T$
and starting with symbols from the same $Q^+$.
The cases for odd and even length are almost identical; we present the even length case.
Without loss of generality, take $w = s_1 t_1 \ldots s_k t_k$ and $w' = s_1' t_1' \ldots s_k' t_k'$ with $s_i, s_i'\in Q_1^+$ and $t_i, t_i'\in Q_2^+$.
If $w$ and $w'$ represent different elements of $S\star T$, then we must have $\ol{s_i}\neq \ol{s_i'}$ or $\ol{t_i}\neq \ol{t_i'}$ for some $i$.
Let us suppose the former, the other case being very similar.
Let $\alpha\in A^*$ with $\alpha\cdot s_i \neq \alpha\cdot s_i'$ and let $\beta = (\$\#)^{i-1} \alpha (\$\#)^{k-i+1}$.  Then
\begin{align*}
\beta\cdot w &= (\$^\circ \#^\circ)^{i-1} (\alpha\cdot s_i)^\circ (\$^\circ \#^\circ)^{k-i} \$^\circ \#,\\
\beta\cdot w' &= (\$^\circ \#^\circ)^{i-1} (\alpha\cdot s_i')^\circ (\$^\circ \#^\circ)^{k-i} \$^\circ \#
\end{align*}
and so $w\neq w'$ in $\Sigma(\A)$.  Hence $\Sigma(\A) = S\star T$.
\end{proof}

The next theorem shows that free products of automaton semigroups are very close to being automaton semigroups.

\begin{theorem}\label{adj}
Let $S$ and $T$ be automaton semigroups.  Then $(S\star T)^1$ is an automaton semigroup.
\end{theorem}
\begin{proof}[Proof sketch.]
The idea is very similar to the previous proposition, but instead of the states $1_S$ and $1_T$ we have a single state $1$,
which acts as the identity on all strings.

We construct an automaton $\A$ with $\Sigma(\A) = (S\star T)^1$ as follows:
Let $\A_1 = (Q_1,A,\delta_1)$ and $\A_2 = (Q_2,B,\delta_2)$ be automata for $S$ and $T$ respectively
and let the alphabet of $\A$ be $C = A\cup B\cup A^\circ \cup B^\circ \cup \{\$,\#,\$^\circ,\#^\circ\}$,
with $A^\circ$ and $B^\circ$ defined as in the previous proof.
Let the set of states of $\A$ be $S\cup T\cup \{1\}$.
Define the transition function $\delta: Q\times C\rightarrow Q\times C$ by $\delta(1,c) = (1,c)$ for all $c\in C$,
$\delta(q,x^\circ) = (q,x^\circ)$ for all $q\in Q$, $x\in A\cup B\cup \{\$,\#\}$ and
\begin{align*}
\delta(s,a)       &= \delta_1(s,a)    &  \delta(t,b)     &= \delta_2(t,b)\\
\delta(s,b)       &= (s,b^\circ)      &  \delta(t,a)     &= (t,a^\circ)\\
\delta(s,\#)      &= (s,\#^\circ)     &  \delta(t,\$)    &= (t,\#^\circ)\\
\delta(s,\$)      &= (1,\$)           &  \delta(t,\#)    &= (1,\#)
\end{align*}
for $a\in A$, $b\in B$, $s\in Q_1$, $t\in Q_2$.

As before, the actions of $S$ on $(C\setminus \{\$\})^*$ and of $T$ on $(C\setminus \{\#\})^*$ are determined solely by their actions on
$A^*$ and $B^*$ respectively.  For $u\in (C\setminus \{\$\})^*$, $v\in C^\omega$ and $s\in Q_1^+$, we have
$u \$ v\cdot s = (u\cdot s) \$ v$.  So again the subsemigroups generated by $Q_1$ and by $Q_2$ are isomorphic to $S$ and $T$ respectively,
and since the state $1$ acts as the identity on all strings, we conclude that $\A$ defines an action of $(S\star T)^1$ on $C^*$.

We distinguish between elements of $(S\star T)^1$ in essentially the same way as in the previous proof.
\end{proof}

If $S$ and $T$ are semigroups, then $S^1 *_{\mon} T^1 = (S\star T)^1$.  If $M$ is a monoid with indecomposable identity, then 
$S = M\setminus \{1\}$ is a semigroup with $M = S^1$.  Hence we have the following corollary to Proposition~\ref{adj}.

\begin{corollary}
The monoid free product of two automaton monoids with indecomposable identities is an automaton monoid.
\end{corollary}

Our constructions rely heavily on the use of a state acting as an identity or left identity on all strings.  
Since all free semigroups of rank at least $2$ are automaton semigroups, we know that the presence of left identities is not 
essential for a free product of automaton semigroups to be an automaton semigroup, but we conjecture that the class of 
automaton semigroups is not closed under free products.  Indeed we conjecture something a good deal stronger:

\begin{conjecture}
There exist finite semigroups $S$ and $T$ such that $S\star T$ is not an automaton semigroup.
\end{conjecture}

We have some hope that we may be able to prove this conjecture using very small semigroups $S$ and $T$, for example
$S$ the trivial semigroup and $T$ a two-element null semigroup (that is, $T = \{t,z\}$ with all products equal to $z$).

\section{Wreath products}

The wreath product of two automaton semigroups is certainly not always an automaton semigroup, since it need not even be finitely generated.
One way to ensure that a wreath product $S\wr T$ is finitely generated is to require $S$ and $T$ to be monoids, with $T$ finite.
The second author asked in \cite{c_1auto} whether, under these restrictions, taking wreath products of automaton monoids always gives 
an automaton monoid.  The answer is almost certainly no, and in fact wreath products of automaton monoids are probably almost never automaton monoids.  

For monoids $S$ and $T$ with $T = \{t_1,\ldots,t_n\}$ finite, the \emph{wreath product} $S\wr T$ of $S$ with $T$ is a semidirect product $S^{|T|}\rtimes T$,
where $T$ acts on elements of $S^{|T|}$ by $(s_{t_1},s_{t_2},\ldots,s_{t_n})^t = (s_{t_1t}, s_{t_2t}, \ldots, s_{t_nt})$.

\begin{conjecture}
The wreath product $\N_0\wr C_2$ is not an automaton monoid.
\end{conjecture}

The obstruction to constructing an automaton $\A$ with $\Sigma(\A) = S\wr T$ seems to be that the automaton needs to be aware
at all times of the `current element' of $T$, in order for states from $S$ to act correctly.  
By this we mean that if we are computing the action of a word $q_1\ldots q_n$, then in order to know how $q_i$ should act on strings,
if $q_1\ldots q_{i-1} = s_i t_i$ with $s_i\in S$, $t_i\in T$, then we need to know $t_i$.  If we attempt to store this information in the states,
we end up with states which do not act as elements of $S\wr T$ on strings.  If we attempt to store the information in the strings,
the difficulty is that we would like it to occur only once, at the start of the string, but it seems impossible to guarantee this without using 
additional states.  The automaton cannot tell what point it is up to in the string, so it treats all symbols of the type used to encode 
the current $t_i$ as if they were in fact $t_i$, leading seemingly unavoidably to `misdirections', in which the automaton does not act as intended.

Semigroups generated by only some of the states in an automaton are also worthwhile objects of study, considered for example in \cite{grigorchuk_automata}.
They are, of course, simply finitely generated subsemigroups of automaton semigroups (since we can always add any finite set of elements of a 
semigroup to the state set for the automaton).  With this point of view, our first obstruction to closure under wreath products falls away.

\begin{theorem}\label{wreathsub}
Let $S$ and $T$ be automaton monoids with $T$ finite.  Then $S\wr T$ is a finitely generated subsemigroup of an automaton semigroup.
\end{theorem}
\begin{proof}
Let $S = \Sigma(\A)$ with $\A = (Q,A,\delta)$ be an automaton monoid and let $T = \{t_1,\ldots,t_n\}$ be a finite monoid.
We construct an automaton $\B$ such that $S\wr T$ is a subsemigroup of $\Sigma(\B)$.
The automaton $\B$ has state set $Q_1^n\cup Q_2^n\cup T$, where each $Q_i$ is a copy of $Q$; and alphabet $A^n\cup B$, where $B$ is a copy of $T$.
For $q\in Q^n$, we denote the copy of $q$ in $Q_i$ by $q_i$.

In a state $t\in T$, the automaton remains in state $T$, not altering the input string until a symbol $b\in B$ is read, at which point it outputs $\ol{bt}$
and moves to state $1_T$, hence leaving the remainder of the string unchanged.
The states in $Q_2^n$ act on symbols in $A^n$ in exactly the same way as in the standard automaton for the direct product $S^n$ (see \cite[Proposition~5.5]{c_1auto}),
ignoring symbols from $B$.
A state $s\in Q_1^n$ ignores symbols from $A^n$, and moves to the state corresponding to $s^b$ in $Q_2^n$ upon reading a symbol $b\in B$ (leaving $b$ unchanged).

All symbols from $A^n$ before the first symbol from $B$, and all subsequent symbols from $B$, are ignored by all states.  Hence the action of $\B$ on $(A^n\cup B)^*$
is completely determined by its action on $B(A^n)^\omega$.  For $\alpha\in (A^n)^*$, $b\in B$, $q_i\in Q_i^n$ $(i=1,2)$ and $t\in T$, we have
\[b\alpha\cdot q_1 = b (\alpha\cdot q^b), \;
b\alpha\cdot q_2 = b (\alpha\cdot q) \; \hbox{and} \;
b\alpha\cdot t = \ol{bt}\alpha.\]

Using these facts, it is straightforward to show that $S\wr T$ is isomorphic to the subsemigroup of $\Sigma(\A)$ generated by $Q_1^n\cup T$.
\end{proof}

We might choose to remove the second obstruction instead, by allowing ourselves to restrict the set of strings the automaton acts on.
The proof of Theorem~\ref{wreathsub} gives an indication of one way to do this.
We define an \emph{initial-symbol automaton semigroup} to be a semigroup which is obtained from an automaton $\A = (Q,B,\delta)$ in
the same way as the automaton semigroup $\Sigma(\A)$, except that we only consider the action on strings in $C D^\omega$, where 
$B$ is the disjoint union of $C$ and $D$.  This is not so far from an automaton semigroup, as $C D^\omega$ can still be viewed as 
a rooted (almost regular) tree, but with the root having a (potentially) different degree to the remaining vertices.  Such a tree is a very natural 
structure for a wreath product to act on.

\begin{theorem}\label{init}
Let $S$ and $T$ be automaton monoids with $T$ finite.  Then $S\wr T$ is an initial-symbol automaton monoid.
\end{theorem}
\begin{proof}
If in the proof of the previous theorem, we restrict the automaton $\B$ to act only on strings in $BA^\omega$, then one copy of $Q^n$ suffices in the 
state set, since the purpose of the first copy was simply to record whether we have yet encountered a symbol from $B$ in processing the string.  
We can thus obtain $S\wr T$ as an initial-automaton monoid by using an automaton similar to the automaton $\B$ in the previous proof, except that it
has states $Q^n$ in place of $Q_1^n\cup Q_2^n$, and these states act like their corresponding versions in $Q_1^n$ on $B$, and like their corresponding versions
in $Q_2^n$ on $A$.
\end{proof}

It may be worth considering further what kind of restrictions on the strings acted on by an automaton give rise to interesting classes of semigroups.  

\section{Further constructions}

A semigroup $S$ is a \emph{small extension} of another semigroup $T$ if $T\leq S$ and $|S\setminus T|$ is finite.
It is easy to see that in this case $S$ being an automaton semigroup need not imply that $T$ is an automaton semigroup, since 
the free monoid of rank $1$ is an automaton semigroup \cite[Proposition~4.4]{c_1auto}, 
while the free semigroup of rank $1$ is not \cite[Proposition~4.3]{c_1auto}.  
The other direction remains open:

\begin{question}~\label{small}{\rm \cite[Open Problem~5.4]{c_1auto}}
If a semigroup $S$ is a small extension of an automaton semigroup $T$, is $S$ necessarily an automaton semigroup?
\end{question}

We again expect the answer to be no.  One possible counterexample is a certain strong semilattice of two semigroups, with the `lower' semigroup being finite.
Some other open problems include the closure or otherwise of the class of automaton semigroups under Rees matrix constructions,
and whether if $S$ is a semigroup such that adjoining a zero to $S$ results in an automaton semigroup, then $S$ itself must be an automaton semigroup.

\section*{Acknowledgements}

The first author's research was funded by an EPSRC grant EP/H011978/1.

The second author's research was funded by the European
Regional Development Fund through the programme {\sc COMPETE} and by
the Portuguese Government through the {\sc FCT} (Funda\c{c}\~{a}o
para a Ci\^{e}ncia e a Tecnologia) under the project {\sc
PEst-C}/{\sc MAT}/{\sc UI0}144/2011 and through an {\sc FCT}
Ci\^{e}ncia 2008 fellowship.

\bibliography{\jobname}
\bibliographystyle{abbrv}

\end{document}